\documentclass[10pt,reqno]{amsart}
\usepackage{amsmath,amssymb,amsfonts,amsthm,mathrsfs,enumitem}
\allowdisplaybreaks

\newtheorem{thm}{Theorem}[section]

\newtheorem{cor}[thm]{Corollary}
\newtheorem{lem}[thm]{Lemma}
\newtheorem{prop}[thm]{Proposition}

\theoremstyle{definition}
\newtheorem{defn}[thm]{Definition}
\theoremstyle{remark}
\newtheorem{rem}[thm]{Remark}
\newtheorem{exa}[thm]{Example}

\numberwithin{equation}{section}

\newcommand{\Z}{\mathbb Z}
\newcommand{\N}{\mathbb N}
\newcommand{\eps}{\varepsilon}
\newcommand{\Sub}{\operatorname{Sub}}
\newcommand{\one}{\mathbf 1}
\newcommand{\bits}{\{0,1\}^{*}}
\newcommand{\wt}{\operatorname{wt}}
\newcommand{\zz}{\operatorname{z}}
\newcommand{\oo}{\operatorname{o}}
\newcommand{\W}{\mathcal W}
\newcommand{\Lcal}{\mathcal L}
\newcommand{\Hcal}{\mathcal H}
\newcommand{\Dcal}{\mathcal D}
\newcommand{\Ccal}{\mathcal C}
\newcommand{\Bcal}{\mathcal B}
\newcommand{\Fcal}{\mathcal F}

\begin{document}

\title[Deficit bounds and equality cases in the Tu--Deng problem]
{Deficit bounds and equality cases in the Tu--Deng problem}
\author{Kaimin Cheng}
\address{School of Mathematical Sciences, China West Normal University, Nanchong 637002, P. R. China}
\email{ckm20@126.com}
\subjclass[2020]{Primary 05A20, 11A63; Secondary 68R05, 94A60}
\keywords{Tu--Deng conjecture, equality cases, binary words, subsequences, Hamming weight}
\date{}

\begin{abstract}
Let $N=2^k-1$, and let $S_{t,k}$ consist of the pairs
$0\le a,b<N$ such that $a+b\equiv t\pmod N$ and
$\wt(a)+\wt(b)<k$.  The Tu--Deng bound $|S_{t,k}|\le 2^{k-1}$ has
recently been proved. In this paper, we give a combinatorial proof of
the equality criterion.  If the $k$-bit cyclic word of $t$ has $Z$
zeros and $g_1,\ldots,g_Z$ are the numbers of ones between successive
zeros, then
$$|S_{t,k}|=2^{k-1}
\quad\Longleftrightarrow\quad
 g_i\ge Z-1\quad(1\le i\le Z).$$
Moreover, we recover the resulting enumeration of the equality parameters.
Beyond equality, if $R\ge Z\ge2$, where $R$ is the number of ones, then
every nonequality parameter satisfies
$$2^{k-1}-|S_{t,k}|\ge 2^{R-Z+1},$$
and we classify all cases in which this bound is attained.  For $R<Z$
we obtain a congruence for $|S_{t,k}|$ and a lower bound for the deficit
in terms of the number of cyclic runs of ones.
\end{abstract}

\maketitle

\section{Introduction}

Let $k\ge2$, put $N=2^k-1$, and let $\wt(n)$ denote the binary Hamming
weight of $n$.  For $1\le t\le N-1$, define
\[
 S_{t,k}
 =
 \bigl\{(a,b)\in\{0,\ldots,N-1\}^2:
 a+b\equiv t\pmod N,\ 
 \wt(a)+\wt(b)<k
 \bigr\}.
\]
The Tu--Deng conjecture asserts
\begin{equation}\label{eq:TD}
 |S_{t,k}|\le 2^{k-1}.
\end{equation}
It arose in the study of Boolean functions with optimal algebraic immunity
\cite{TuDeng2011}.  Flori, Randriambololona, Cohen and Mesnager
\cite{FloriRandriamCohenMesnager2010} found a family attaining equality and
conjectured that these are all equality cases.  Deng and Yuan
\cite{DengYuan2012} later stated the corresponding converse problem
explicitly.  Further partial results include
\cite{CusickLiStanica2011,ChengHongZhong2015,ChenLinWei2020}, while
Spiegelhofer and Wallner \cite{SpiegelhoferWallner2019} proved
\eqref{eq:TD} for almost all parameters and showed that it implies
Cusick's sum-of-digits conjecture.

Complete proofs of \eqref{eq:TD} appeared in 2026 by Liu, Luo and Xie
\cite{LiuLuoXie2026}, Cusick \cite{Cusick2026}, and Yuan, Hu, Zhang,
Zhang and Wu \cite{YuanEtAl2026}.  The first two proofs use, respectively,
a language of binary subsequences and a cyclic matrix polynomial.  The
third proof uses a different integral argument based on pivotal groups.
Yuan et al.\ also obtained the equality characterization and its
enumeration.  The present paper gives a different proof of that
characterization, identifies the coefficient arrays occurring in the
proofs of Liu--Luo--Xie and Cusick, and then studies the deficit from
the Tu--Deng bound.

Regard the $k$-bit word of $t$ as a cyclic word.  Let $R$ and $Z$ be its
numbers of ones and zeros, respectively, and let $g_1,\ldots,g_Z$ be the
numbers of ones between successive zeros.  Thus
\[
 g_1+\cdots+g_Z=R.
\]
For example, the cyclic word
\[
 0\,11\,0\,111\,0\,11
\]
has $Z=3$ and $(g_1,g_2,g_3)=(2,3,2)$.

Our main equality statement is the following.
\begin{thm}\label{thm:intro-equality}
Let $1\le t<2^k-1$.  Then
\[
 |S_{t,k}|=2^{k-1}
 \quad\Longleftrightarrow\quad
 g_i\ge Z-1\quad(1\le i\le Z).
\]
\end{thm}
The sufficiency of this condition was known from
\cite{FloriRandriamCohenMesnager2010}, and the converse was their
Conjecture~3.20.  The same characterization, together with the
enumeration derived from it, was also obtained in
\cite[Theorem~3.1.2 and Corollary~3.1.1]{YuanEtAl2026}.  In that paper a
cyclic gap is the distance between two consecutive zero positions, so
their condition that every gap is at least $Z$ is equivalent to the
run-length condition $g_i\ge Z-1$ used here.  The proof given here is
different and is useful for the quantitative questions below.

We also study how far a nonequality parameter must lie below the
Tu--Deng bound.  If $R\ge Z\ge2$, then
\[
 2^{k-1}-|S_{t,k}|\ge 2^{R-Z+1},
\]
and we determine exactly when equality holds in this estimate.  When
$R<Z$, we prove that equality in the Tu--Deng bound is impossible,
obtain a congruence for $|S_{t,k}|$, and give a lower bound for the
deficit in terms of the number of cyclic runs of ones.

The proof first identifies the coefficients used by Liu--Luo--Xie and
Cusick.  We then use a marked-deletion count and a direct description of
subsequences by run lengths.  Section~\ref{sec:prelim} recalls the two
coefficient formulations, Section~\ref{sec:bridge} identifies them,
Sections~\ref{sec:onesided}--\ref{sec:top} develop the combinatorial
tools, and Section~\ref{sec:equality} proves the equality and deficit
results.

\section{Two formulations of the Tu--Deng count}\label{sec:prelim}

For a finite binary word $u$, let $|u|_0$ and $|u|_1$ denote the
numbers of $0$'s and $1$'s in $u$, respectively, and write
\[
\zz(u)=|u|_0,\qquad \oo(u)=|u|_1.
\]
Throughout, $\N=\{0,1,2,\ldots\}$.  Whenever a binary word represents
an integer, we follow the least-significant-bit-first convention of
\cite{LiuLuoXie2026}; cyclic rotations are always taken in this reading
direction. We use the lexicographic order determined by $0<1$, with a
proper prefix smaller than the word that it prefixes.

\subsection{The Liu--Luo--Xie language}

For finite binary words $u$ and $v$, write $u\preceq v$ if $u$ is a
subsequence of $v$, and set
\[
 \Sub(v)=\{u:u\preceq v\}.
\]
For convenience, write
\[
 D(w)=\Sub(w)
\]
for the subsequence language generated by a finite binary word $w$.
For $b\in\{0,1\}$, define the directional boundary
\[
 \partial_bD(w)
 =
 \{ub:u\in D(w),\ ub\notin D(w)\}.
\]
Its bivariate weight enumerator is
\[
 B_w^b(X,Y)
 =
\sum_{y\in\partial_bD(w)}
X^{\zz(y)}Y^{\oo(y)}.
\]

Following \cite{LiuLuoXie2026}, for a word $v$ define
\[
 \Hcal(v)
 =
 \{y\in\partial_1D(v):y<_{\rm lex}v\},
 \qquad
 \Lcal(v)
 =
D(v)\mathbin{\dot\cup}\Hcal(v).
\]
The bivariate weight enumerator of $\Lcal(v)$ is
\[
 J_v(X,Y)
 =
\sum_{y\in\Lcal(v)}X^{\zz(y)}Y^{\oo(y)}
 =
\sum_{a,b\ge0}j_{a,b}(v)X^aY^b,
\]
so that $j_{a,b}(v)$ is the number of words in $\Lcal(v)$ containing
exactly $a$ zeros and $b$ ones.

For the remainder of this subsection, fix $1\le t<2^k-1$ and choose
a cyclic rotation of the $k$-bit word of $t$ of the form
$W=10v$. Put
\[
R=\oo(W),\qquad Z=\zz(W).
\]
The cyclic-carry enumerator of \cite{LiuLuoXie2026} is
\[
 \mathscr C_v(X,Y)
 =
B_{0v}^1(X,Y)+B_v^0(X,Y).
\]

We shall use the following facts from \cite{LiuLuoXie2026}.

\begin{prop}\label{prop:LLX}
The following statements hold.
\begin{enumerate}[label=\textup{(\roman*)},leftmargin=2.8em]
\item
The language $\Lcal(v)$ is closed under deletion of any single $1$.

\item
If
\[
 \Lambda=X+Y-1,
\]
then
\begin{equation}\label{eq:LLX-factor}
 \mathscr C_v(X,Y)
 =
 1+\Lambda J_v(X,Y)+X^ZY^R.
\end{equation}

\item
Define
\[
 \Phi_-\!\left(\sum_{a,b\ge0}c_{a,b}X^aY^b\right)
 =
\sum_{a>b}c_{a,b}2^{-a-b}.
\]
Then
\begin{equation}\label{eq:halfplane}
\frac{|S_{t,k}|}{2^k}
 =
 \Phi_-(\mathscr C_v).
\end{equation}
\end{enumerate}
\end{prop}

For $\ell\ge1$, put
\[
 \Delta_\ell(v)
 =
 2j_{\ell,\ell-1}(v)-j_{\ell,\ell}(v).
\]
By the deletion inequality of \cite{LiuLuoXie2026}, one has
\[
\Delta_\ell(v)\ge0.
\]

\begin{prop}\label{prop:deficit}
The following identity holds:
\begin{equation}\label{eq:deficit-general}
 2^{k-1}-|S_{t,k}|+\one_{\{Z>R\}}
 =
 \sum_{\ell\ge1}2^{k-2\ell-1}\Delta_\ell(v).
\end{equation}
Moreover,
\[
 \Delta_\ell(v)=0
 \qquad\text{for }\ell>\min\{R,Z-1\}.
\]
\end{prop}

\begin{proof}
For a monomial $X^aY^b$, direct evaluation gives
\[
 \Phi_-(\Lambda X^aY^b)
 =
 \begin{cases}
 2^{-2\ell-1},&a=b=\ell,\\
 -2^{-2b-2},&a=b+1,\\
 0,&\text{otherwise}.
 \end{cases}
\]
Since $j_{0,0}=1$, it follows that
\[
 \Phi_-(\Lambda J_v)
 =
 \frac12
 -
 \sum_{\ell\ge1}2^{-2\ell-1}\Delta_\ell(v).
\]
The exceptional monomial $X^ZY^R$ contributes $2^{-k}$ if $Z>R$
and contributes $0$ otherwise.  Combining this with
\eqref{eq:LLX-factor} and \eqref{eq:halfplane}, and multiplying by
$2^k$, gives \eqref{eq:deficit-general}.

Now $v$ contains $Z-1$ zeros.  Every word of $\Lcal(v)$ also has at
most $Z-1$ zeros, because the extra boundary letter is a $1$.
Hence $j_{\ell,\ell}=j_{\ell,\ell-1}=0$ for $\ell>Z-1$.  If $R<Z$,
then $v$ has $R-1$ ones, and Lemma~5.4 of \cite{LiuLuoXie2026} gives
$\deg_YJ_v\le R-1$.  Thus $j_{\ell,\ell}=0$ for $\ell\ge R$ and
$j_{\ell,\ell-1}=0$ for $\ell>R$, proving the last assertion.
\end{proof}

Factoring out the smallest power of $2$ gives the following formulas.

\begin{cor}\label{cor:divisibility}
If $R\ge Z$ and $m=Z-1$, then
\begin{equation}\label{eq:deficit-factor-plus}
 2^{k-1}-|S_{t,k}|
 =
 2^{R-Z+1}
 \sum_{\ell=1}^{m}4^{m-\ell}\Delta_\ell(v).
\end{equation}
If $R<Z$, then
\begin{equation}\label{eq:deficit-factor-minus}
 2^{k-1}-|S_{t,k}|+1
 =
 2^{Z-R-1}
 \sum_{\ell=1}^{R}4^{R-\ell}\Delta_\ell(v).
\end{equation}
\end{cor}

\begin{proof}
This follows by factoring the smallest power of $2$ from
\eqref{eq:deficit-general}.  When $R\ge Z=1$, the sum in
\eqref{eq:deficit-factor-plus} is empty and both sides are zero.
\end{proof}

\subsection{Cusick's cyclic polynomial}

For variables $u,v$, Cusick \cite{Cusick2026} uses
\[
 C_0(u,v)
 =
\begin{pmatrix}1&0\\u&v\end{pmatrix},
 \qquad
 C_1(u,v)
 =
\begin{pmatrix}u&v\\0&1\end{pmatrix}.
\]
For a cyclic word
\[
W=\eps_0\eps_1\cdots\eps_{k-1}
\]
with $R$ ones and $Z$ zeros, put
\[
 M_W(u,v)
 =
 C_{\eps_0}(u,v)\cdots C_{\eps_{k-1}}(u,v).
\]
Cusick defines $H_W$ by
\begin{equation}\label{eq:Cusick-H}
\operatorname{tr}M_W(u,v)-u^Rv^Z
 =
 1+(u+v-1)H_W(u,v),
\end{equation}
and writes
\[
 H_W(u,v)
 =
\sum_{p,q\ge0}h_{p,q}(W)u^pv^q.
\]
He proves that
\[
 h_{p,q}(W)\in\mathbb Z_{\ge0}
\qquad(p,q\ge0),
\]
and, in particular,
\[
 h_{\ell,\ell}(W)
 \le
 2h_{\ell-1,\ell}(W)
\qquad(\ell\ge1).
\]

\section{Identifying the two coefficient formulations}\label{sec:bridge}

The purpose of this section is to show that the coefficients used in the
two recent proofs count the same objects.  Once this is established, we
will use the language \(\Lcal(v)\), where deletion has a direct
combinatorial meaning.

We first rewrite the Liu--Luo--Xie boundary recurrence in matrix form.
Recall the polynomials $B_w^0(X,Y)$ and $B_w^1(X,Y)$ defined in
Section~\ref{sec:prelim}, and set
\[
 B_w
 =
\begin{pmatrix}
 B_w^0\\
 B_w^1
 \end{pmatrix}.
\]
For the empty word $\eps$,
\[
 B_\eps
 =
\begin{pmatrix}
 X\\
 Y
 \end{pmatrix}.
\]

Following \cite{LiuLuoXie2026}, define
\[
 A_0(X,Y)
 =
\begin{pmatrix}
 X&0\\
 Y&1
 \end{pmatrix},
 \qquad
 A_1(X,Y)
 =
\begin{pmatrix}
 1&X\\
 0&Y
 \end{pmatrix}.
\]
Then, for $w=w_1\cdots w_n$, the directional-boundary recurrence takes
the matrix form
\begin{equation}\label{eq:B-recurrence}
 B_w
 =
 A_{w_n}\cdots A_{w_1}B_\eps.
\end{equation}
Recall also that
\[
 \mathscr C_v(X,Y)
 =
B_{0v}^1(X,Y)+B_v^0(X,Y).
\]

Finally, put
\[
 P=
\begin{pmatrix}
 0&1\\
 1&0
 \end{pmatrix}.
\]

\begin{lem}\label{lem:matrix-conj}
For $b\in\{0,1\}$,
\[
 C_b(Y,X)=P A_b(X,Y)^T P.
\]
\end{lem}

\begin{proof}
This is a direct multiplication.
\end{proof}

For $v=v_1\cdots v_n$, write
\[
 T_v=A_{v_n}\cdots A_{v_1}.
\]

\begin{lem}\label{lem:trace-carry}
The Liu--Luo--Xie carry polynomial satisfies
\[
 \mathscr C_v(X,Y)
 =
 \operatorname{tr}(T_vA_0A_1).
\]
\end{lem}

\begin{proof}
By \eqref{eq:B-recurrence},
\[
 B_v=T_vB_\eps,
 \qquad
 B_{0v}=T_vA_0B_\eps.
\]
Writing
\[
 T_v=
 \begin{pmatrix}p&q\\r&s\end{pmatrix},
\]
we have
\[
 B_v^0=pX+qY
\]
and
\[
 B_{0v}^1=rX^2+s(XY+Y).
\]
On the other hand,
\[
 A_0A_1
 =
 \begin{pmatrix}
 X&X^2\\
 Y&XY+Y
 \end{pmatrix},
\]
so
\[
 \operatorname{tr}(T_vA_0A_1)
 =
 pX+qY+rX^2+s(XY+Y)
 =
 B_v^0+B_{0v}^1
 =
 \mathscr C_v,
\]
as desired.
\end{proof}

\begin{thm}\label{thm:bridge}
Let $W=10v$. Then we have
\begin{equation}\label{eq:bridge}
H_W(Y,X)=J_v(X,Y).
\end{equation}
Equivalently, for all $j,q\ge0$,
\begin{equation}\label{eq:bridge-coeff}
 h_{q,j}(W)
 =
 j_{j,q}(v)
 =
 |\{y\in\Lcal(v):\zz(y)=j,\ \oo(y)=q\}|.
\end{equation}
\end{thm}

\begin{proof}
Write $v=v_1\cdots v_n$.  Lemma~\ref{lem:matrix-conj} gives
\[
 M_{10v}(Y,X)
 =
 P\bigl(A_{v_n}\cdots A_{v_1}A_0A_1\bigr)^TP
 =
 P(T_vA_0A_1)^TP.
\]
Therefore
\[
 \operatorname{tr}M_W(Y,X)
 =
 \operatorname{tr}(T_vA_0A_1)
 =
 \mathscr C_v(X,Y),
\]
where the chosen reading convention and cyclic invariance of trace fix the
orientation unambiguously.  By \eqref{eq:Cusick-H},
\[
 \operatorname{tr}M_W(Y,X)
 =
 1+(X+Y-1)H_W(Y,X)+X^ZY^R.
\]
By \eqref{eq:LLX-factor},
\[
 \mathscr C_v(X,Y)
 =
 1+(X+Y-1)J_v(X,Y)+X^ZY^R.
\]
Subtracting and cancelling the nonzero polynomial $X+Y-1$ in
$\Z[X,Y]$ proves \eqref{eq:bridge}.  Comparing coefficients gives
\eqref{eq:bridge-coeff}.
\end{proof}

By Theorem~\ref{thm:bridge},
\[
 \Delta_\ell(v)=2h_{\ell-1,\ell}(W)-h_{\ell,\ell}(W).
\]
The right-hand side depends only on the cyclic word $W$, and from this
point onward we write this common quantity as $\Delta_\ell(W)$.

\begin{rem}
Theorem~\ref{thm:bridge} explains why the coefficient inequality of
\cite{Cusick2026} and the deletion inequality of
\cite{LiuLuoXie2026} have the same shape: they are statements about
the same coefficients.
\end{rem}

\section{Marked deletion and its equality case}
\label{sec:onesided}

We now isolate the deletion argument needed later.  The inequality is
elementary: delete a marked \(1\) and count the possible insertion
slots in the resulting word.  The important point for the present
paper is its equality case.

Let \(\bits\) denote the set of all finite binary words, including the
empty word.

\begin{defn}
A finite binary language $\Lcal\subseteq\bits$ is
\emph{$1$-deletion-closed} if, whenever $w\in\Lcal$ and $w'$ is
obtained from $w$ by deleting a single $1$, then
$w'\in\Lcal$.
\end{defn}

For $a,b\ge0$, put
\[
 \W_{a,b}
 =
 \{w\in\bits:|w|_0=a,\ |w|_1=b\},
 \qquad
 \Lcal_{a,b}
 =
 \Lcal\cap\W_{a,b}.
\]

\begin{prop}\label{prop:marked}
If $\Lcal$ is $1$-deletion-closed and $b\ge1$, then
\begin{equation}\label{eq:marked}
 b|\Lcal_{a,b}|
 \le
 (a+b)|\Lcal_{a,b-1}|.
\end{equation}
\end{prop}

\begin{proof}
Mark one of the $b$ ones of a word in $\Lcal_{a,b}$ and delete it.
This gives $b|\Lcal_{a,b}|$ marked deletion objects.  The image
consists of a word $z\in\Lcal_{a,b-1}$ together with one of the
$a+b$ insertion slots for a marked $1$.  The map is injective, since
the target word and the marked slot recover the source word and the
marked occurrence.  This proves \eqref{eq:marked}.
\end{proof}

\begin{thm}\label{thm:deletion-equality}
Let $\Lcal$ be a finite $1$-deletion-closed language and $b\ge1$.
Equality in \eqref{eq:marked} holds if and only if either
\[
 \Lcal_{a,b}=\Lcal_{a,b-1}=\varnothing,
\]
or
\[
 \Lcal_{a,b}=\W_{a,b},
 \qquad
 \Lcal_{a,b-1}=\W_{a,b-1}.
\]
\end{thm}

\begin{proof}
If equality holds in Proposition~\ref{prop:marked}, every insertion
slot above every word of $\Lcal_{a,b-1}$ must occur in the image.

If $\Lcal_{a,b-1}=\varnothing$, then $\Lcal_{a,b}=\varnothing$ by
$1$-deletion closure.  Suppose that $\Lcal_{a,b-1}\ne\varnothing$.
Then $\Lcal_{a,b}$ is nonempty.  We claim it is invariant under the
adjacent exchange $10\leftrightarrow01$.  Indeed, if
$w=x10y\in\Lcal_{a,b}$,
delete the displayed $1$.  The word $x0y$ belongs to
$\Lcal_{a,b-1}$; equality in the marked-deletion count says that
reinserting a marked $1$ in the slot after the displayed $0$ is
admissible, hence
$x01y\in\Lcal_{a,b}$. The reverse exchange is obtained in the same way.  The graph on
$\W_{a,b}$ generated by adjacent exchanges of unlike letters is
connected, so
$\Lcal_{a,b}=\W_{a,b}$.
Every word of $\W_{a,b-1}$ is obtained by deleting a $1$ from some
word of $\W_{a,b}$, hence $1$-deletion closure gives
$\Lcal_{a,b-1}=\W_{a,b-1}$.
The converse follows directly.
\end{proof}

\section{Describing the subsequence language by run lengths}\label{sec:gapcoords}

We next describe the words in \(\Lcal(v)\) with prescribed numbers of zeros and ones by using the lengths of their runs of ones.  The construction has a direct meaning: choose the
zeros of \(v\) that are retained in the subsequence; every skipped zero
merges two neighboring intervals of ones, so the corresponding
available run lengths are added.

Write
\begin{equation}\label{eq:v-runs}
 v
 =
1^{a_0}0\,1^{a_1}0\cdots0\,1^{a_m},
\end{equation}
so $v$ has exactly $m$ zeros. For $0\le j\le m$ and
$\mathbf x=(x_0,\ldots,x_j)\in\N^{j+1}$, set
\[
 |\mathbf x|=x_0+\cdots+x_j,
 \qquad
 w(\mathbf x)
 =
1^{x_0}0\,1^{x_1}0\cdots0\,1^{x_j}.
\]

\subsection{The subsequence part}

Choose
\[
 I=\{r_1<\cdots<r_j\}\subseteq\{1,\ldots,m\}.
\]
The chosen zeros divide the one-runs in \eqref{eq:v-runs} into
$j+1$ intervals.  The numbers of ones available in these intervals are
\begin{align*}
 c_0(I)
 &=
 \sum_{\nu=0}^{r_1-1}a_\nu,
 \nonumber\\
 c_s(I)
 &=
 \sum_{\nu=r_s}^{r_{s+1}-1}a_\nu
 \qquad(1\le s<j),\\
 c_j(I)
 &=
 \sum_{\nu=r_j}^{m}a_\nu.
 \nonumber
\end{align*}
For $j=0$ there is a single bound
\[
 c_0(\varnothing)=a_0+\cdots+a_m.
\]

\begin{lem}\label{lem:coarse}
Let $\mathbf x=(x_0,\ldots,x_j)\in\N^{j+1}$. Then
$$w(\mathbf x)\preceq v$$
if and only if there exists $I\subseteq\{1,\ldots,m\}$ with
$|I|=j$ such that
$$x_s\le c_s(I)\ \text{for\ every\ }
s.$$
\end{lem}

\begin{proof}
Given an embedding of $w(\mathbf x)$ into $v$, let
$r_1<\cdots<r_j$ be the selected zeros used for the $j$ zeros of
$w(\mathbf x)$.  Between two consecutive selected zeros, every
unselected zero is skipped, so the corresponding block in the subsequence can use
ones from all one-runs in that interval.  This gives
$x_s\le c_s(I)$.

Conversely, if these inequalities hold, choose the zeros indexed by
$I$ and, in each interval between consecutive chosen zeros, select
$x_s$ of the $c_s(I)$ available ones.  This constructs a subsequence
embedding.
\end{proof}

\begin{exa}\label{exa:gap-coordinates}
Let
\[
 v=11\,0\,1\,0\,111,
\]
so \((a_0,a_1,a_2)=(2,1,3)\).  Consider subsequences with one zero.
If the first zero of \(v\) is retained, then the available numbers of
ones before and after it are
\[
 c_0=2,\qquad c_1=1+3=4.
\]
Hence every word \(1^{x_0}0\,1^{x_1}\) with
\(x_0\le2\) and \(x_1\le4\) occurs as a subsequence.  If instead the
second zero is retained, the first zero is skipped, so the first two
one-runs merge.  The bounds become
\[
 c_0=2+1=3,\qquad c_1=3.
\]
Thus the set of one-zero subsequences is the union of these two
explicit families.  The general construction below records exactly
this operation for an arbitrary choice of retained zeros.
\end{exa}

Define
\[
 \Dcal_{j,q}(\mathbf a)
 =
\bigcup_{\substack{I\subseteq\{1,\ldots,m\}\\|I|=j}}
 \left\{
 \mathbf x\in\N^{j+1}:
 |\mathbf x|=q,\ 
 x_s\le c_s(I)\ \text{for\ every\ }
s
 \right\}.
\]
Lemma~\ref{lem:coarse} says that
\[
 \mathbf x\longmapsto w(\mathbf x)
\]
is a bijection from $\Dcal_{j,q}(\mathbf a)$ onto
$D(v)\cap\W_{j,q}$.

\subsection{The lexicographic boundary condition}

It remains to describe the additional words in $\Hcal(v)$.  These are
boundary words whose lexicographic order relative to $v$ is part of the
definition, and the next lemma translates that condition into the run
coordinates.

\begin{lem}\label{lem:lex-condition}
Let $\mathbf x\in\N^{j+1}$ and suppose
\[
 w(\mathbf x)\not\preceq v.
\]
Then
\[
 w(\mathbf x)<_{\rm lex}v
\]
if and only if there exists $s\in\{0,\ldots,j-1\}$ such that
\begin{equation}\label{eq:lex-condition}
 x_0=a_0,\ldots,x_{s-1}=a_{s-1},
 \qquad
 x_s<a_s.
\end{equation}
\end{lem}

\begin{proof}
Compare the two words from the left.  If the first discrepancy occurs
inside the $s$th one-run before the $s$th target zero, then
$x_s<a_s$ means that $w(\mathbf x)$ reaches a $0$ while $v$ still has
a $1$, hence $w(\mathbf x)<_{\rm lex}v$; the opposite inequality gives
the reverse lexicographic order.

If the first $j$ one-runs agree exactly, then a shorter final run makes
$w(\mathbf x)$ a prefix of $v$ and hence a subsequence, contrary to
the hypothesis.  An equal final run is again a subsequence, while a
longer final run makes $w(\mathbf x)>_{\rm lex}v$.  Thus the only
possible way to have $w(\mathbf x)<_{\rm lex}v$ outside $D(v)$ is
\eqref{eq:lex-condition}.
\end{proof}

Let $\Lambda_j(\mathbf a)$ denote the set of all $\mathbf x$ satisfying
\eqref{eq:lex-condition} for some $s<j$. Let $e_j=(0,\ldots,0,1)\in\N^{j+1}$ denote the last standard basis vector. Define
\[
 \Bcal_{j,q}(\mathbf a)
 =
 \left\{
 \mathbf x\in\N^{j+1}\ \middle|\
 \begin{aligned}
 &|\mathbf x|=q,\quad x_j\ge1,\\
 &\mathbf x-e_j\in\Dcal_{j,q-1}(\mathbf a),\quad
   \mathbf x\notin\Dcal_{j,q}(\mathbf a),\\
 &\mathbf x\in\Lambda_j(\mathbf a)
 \end{aligned}
 \right\}.
\]
\begin{prop}\label{prop:boundary-coarse}
The map $\mathbf x\mapsto w(\mathbf x)$ is a bijection from
$\Bcal_{j,q}(\mathbf a)$ onto
\[
 \Hcal(v)\cap\W_{j,q}.
\]
\end{prop}

\begin{proof}
A word $w(\mathbf x)$ belongs to $\partial_1D(v)$ if and only if it
ends in $1$, deleting that final $1$ gives a word of $D(v)$, and the
full word is not in $D(v)$.  In these coordinates the conditions are
exactly
\[
 x_j\ge1,\qquad
 \mathbf x-e_j\in\Dcal_{j,q-1}(\mathbf a),
 \qquad
 \mathbf x\notin\Dcal_{j,q}(\mathbf a).
\]
The additional condition defining $\Hcal(v)$ is
$w(\mathbf x)<_{\rm lex}v$, which is exactly the condition
$\mathbf x\in\Lambda_j(\mathbf a)$ by
Lemma~\ref{lem:lex-condition}.
\end{proof}

Set
\[
 \Ccal_{j,q}(\mathbf a)
 =
 \Dcal_{j,q}(\mathbf a)
 \mathbin{\dot\cup}
 \Bcal_{j,q}(\mathbf a).
\]

\begin{thm}\label{thm:gap-coordinates}
For every $0\le j\le m$ and $q\ge0$, the map
\[
 \mathbf x\longmapsto
 1^{x_0}0\,1^{x_1}0\cdots0\,1^{x_j}
\]
is a bijection
\[
 \Ccal_{j,q}(\mathbf a)
 \longrightarrow
 \Lcal(v)\cap\W_{j,q}.
\]
Consequently, if $W=10v$, then
\begin{equation}\label{eq:coarse-h}
 h_{q,j}(W)
 =
 |\Ccal_{j,q}(\mathbf a)|.
\end{equation}
\end{thm}

\begin{proof}
The first assertion is the disjoint union of
Lemma~\ref{lem:coarse} and Proposition~\ref{prop:boundary-coarse}.
The coefficient identity follows from Theorem~\ref{thm:bridge}.
\end{proof}

\begin{rem}
The coordinates depend on the chosen cut.  Choose an occurrence of \(10\)
in the cyclic word \(W\) and write \(W=10v\).
A different cut changes the run coordinates \(\mathbf a\) and the
lexicographic boundary condition, but not the cardinalities in
\eqref{eq:coarse-h}, because \(H_W\) is invariant under cyclic
rotation.
\end{rem}

\section{The level with $Z-1$ zeros}
\label{sec:top}

At the highest relevant zero level there is no choice of retained
zeros: every zero of \(v\) must be used.  The description from the previous section therefore reduces to a single
set of integer vectors with coordinatewise upper bounds.  This specialization is the part needed for the equality theorem and for the deficit estimate in the range
$R\ge Z$.

Assume throughout this section that
\[
 R\ge Z\ge2,
 \qquad
 m=Z-1.
\]
Choose a cyclic rotation $W=10v$ and write
\[
 v=1^{a_0}0\,1^{a_1}0\cdots0\,1^{a_m}.
\]
Let $g_0,\ldots,g_m$ be the gap lengths of $W$, ordered starting from the zero in the displayed occurrence of $10$, and put $\mathbf g=(g_0,\ldots,g_m)$. Then
\begin{equation}\label{eq:gaps-a}
g_i=a_i\quad(0\le i<m),
 \qquad
 g_m=a_m+1,
\end{equation}
and
\[
g_0+\cdots+g_m=R.
\]

\begin{thm}\label{thm:top-box}
For every $0\le q<R$,
\begin{equation}\label{eq:top-box}
 \Ccal_{m,q}(\mathbf a)
=
\left\{
\mathbf x\in\N^{m+1}:
|\mathbf x|=q,\ 
x_i\le g_i\ \text{for\ every}\ i
\right\}.
\end{equation}
Consequently,
\begin{equation}\label{eq:top-coeff}
 h_{q,m}(W)
 =
 [y^q]
\prod_{i=0}^{m}(1+y+\cdots+y^{g_i}),
\end{equation}
where $[y^q]P(y)$ denotes the coefficient of $y^q$ in $P(y)$.
\end{thm}

\begin{proof}
At level $j=m$ every zero of $v$ must be used. Hence the
subsequence part consists precisely of the vectors satisfying
\[
 x_i\le a_i\qquad(0\le i\le m).
\]
A boundary vector has $x_m\ge1$, and after deleting the final $1$ it
must satisfy these inequalities.  Since the full word is not a
subsequence, this forces
\[
 x_m=a_m+1,
 \qquad
 x_i\le a_i\quad(i<m).
\]
The lexicographic boundary condition excludes only the vector
$(a_0,\ldots,a_{m-1},a_m+1)$,
for which no earlier run is shortened.  Its total degree is
\[
a_0+\cdots+a_m+1=R.
\]
Thus for $q<R$ no exclusion remains, and
\eqref{eq:gaps-a} gives exactly \eqref{eq:top-box}.  Taking the
ordinary generating function proves \eqref{eq:top-coeff}.
\end{proof}

We record the following consequence of Theorem~\ref{thm:top-box}.  For $q=m$, the unrestricted level consists of all vectors
$\mathbf x\in\N^{m+1}$ with $|\mathbf x|=m$.  Hence the level in
\eqref{eq:top-box} is complete if and only if it contains $m e_i$ for
every $i$, which is equivalent to
\[
 g_i\ge m\qquad(0\le i\le m).
\]

\section{Equality, enumeration, and deficit bounds}\label{sec:equality}

For the $k$-bit cyclic word $W$ of $t$, write
\[
 R=\oo(W),\qquad Z=\zz(W).
\]
Since $1\le t<2^k-1$, we always have $Z\ge1$.

\subsection{Equality and enumeration}

We first dispose of the range in which the cyclic word has more zeros
than ones.

\begin{lem}\label{lem:low-no-equality}
If $R<Z$, then
\[
 |S_{t,k}|<2^{k-1}.
\]
\end{lem}

\begin{proof}
Choose a cyclic rotation $W=10v$.  Since $v$ has $R-1$ ones,
Lemma~5.4 of \cite{LiuLuoXie2026} gives
$j_{R,R}(v)=0$.  Since $v$ has $Z-1\ge R$ zeros, it has a subsequence
with $R$ zeros and $R-1$ ones, so $j_{R,R-1}(v)\ge1$.  Hence
\[
 \Delta_R(W)=2j_{R,R-1}(v)-j_{R,R}(v)\ge2.
\]
Using Proposition~\ref{prop:deficit} and $k=R+Z$,
\[
 2^{k-1}-|S_{t,k}|+1
 \ge
 2^{k-2R-1}\Delta_R(W)
 \ge
 2^{Z-R}\ge2.
\]
Thus $2^{k-1}-|S_{t,k}|\ge1$.
\end{proof}

\begin{lem}\label{lem:top-nonempty}
Assume $R\ge Z\ge2$ and put $m=Z-1$.  Then the level
$\Lcal(v)\cap\W_{m,m}$ is nonempty.
\end{lem}

\begin{proof}
The word $v$ has $m$ zeros and $R-1\ge m$ ones.  Choosing all zeros
and any $m$ ones gives a balanced subsequence of $v$.
\end{proof}

\begin{thm}\label{thm:equality}
Let $1\le t<2^k-1$.  Let $Z$ be the number of zeros in its $k$-bit
cyclic word, and let
\[
 g_1,\ldots,g_Z
\]
be the cyclic numbers of ones between successive zeros.  Then
\begin{equation}\label{eq:equality}
|S_{t,k}|=2^{k-1}
\quad\Longleftrightarrow\quad
 g_i\ge Z-1\ \text{for\ every\ }
 1\le i\le Z.
\end{equation}
\end{thm}

\begin{proof}
If $R<Z$, equality is impossible by
Lemma~\ref{lem:low-no-equality}.  Suppose $R\ge Z$.  If $Z=1$, then
the sum in Proposition~\ref{prop:deficit} is empty, so
$|S_{t,k}|=2^{k-1}$; the gap condition is automatic because its unique
gap satisfies $g_1=R\ge0=Z-1$.  Hence assume $Z\ge2$ and put
$m=Z-1$.

Suppose first that equality holds.  By
Proposition~\ref{prop:deficit}, every $\Delta_\ell(W)$ vanishes,
in particular $\Delta_m(W)=0$.  By
Lemma~\ref{lem:top-nonempty} and
Theorem~\ref{thm:deletion-equality}, the level $\Lcal(v)\cap\W_{m,m}$ is the whole of $\W_{m,m}$.  By Theorems~\ref{thm:bridge} and
\ref{thm:top-box}, this level is the set
\[
 \left\{
 \mathbf x\in\N^{m+1}:
 |\mathbf x|=m,\ x_i\le g_i
 \right\}.
\]
Since it is full, it contains the pure vector $m e_i$ for every $i$.
Thus
\[
 g_i\ge m=Z-1
\]
for every cyclic gap.

Conversely suppose all cyclic gaps satisfy $g_i\ge m$.
The level $\Lcal(v)\cap\W_{m,m}$ is full by
Theorem~\ref{thm:top-box}.  We claim that
$\Lcal(v)\cap\W_{\ell,\ell}$ is full for every $\ell<m$.  Choose a
cyclic rotation $W=10v$ and write
$v$ as in \eqref{eq:v-runs}.  Then
\[
 a_0,\ldots,a_{m-1}\ge m,
 \qquad
 a_m\ge m-1.
\]
Fix $\ell<m$.  Any word of $\W_{\ell,\ell}$ has gap exponents
$x_0,\ldots,x_\ell$ with each $x_s\le \ell\le m-1$.  Use the first
$\ell$ zeros of $v$.  For this choice of zeros, the first $\ell$ bounds are at least $m$,
and the last is at least $m-1\ge \ell$.
Hence Lemma~\ref{lem:coarse} shows that every word of
$\W_{\ell,\ell}$ belongs to $D(v)$.

Thus
\[
 \Lcal(v)\cap\W_{\ell,\ell}=\W_{\ell,\ell}
 \qquad(1\le \ell\le m).
\]
Since every word
of $\W_{\ell,\ell-1}$ is obtained by deleting a $1$ from some word of
$\W_{\ell,\ell}$, deletion of a $1$ then gives
$\Lcal(v)\cap\W_{\ell,\ell-1}=\W_{\ell,\ell-1}$ as well.  Hence
\[
 \Delta_\ell(W)=0
 \qquad(1\le \ell\le m).
\]
Equation \eqref{eq:deficit-general} now gives
$|S_{t,k}|=2^{k-1}$.
\end{proof}

\begin{rem}
Flori, Randriambololona, Cohen and Mesnager
\cite[Conjecture~3.20]{FloriRandriamCohenMesnager2010} conjectured that
their separated-zero equality family exhausts all parameters attaining the
bound.  Deng and Yuan \cite{DengYuan2012} later stated explicitly that
the converse to their corresponding equality theorem was conjectured.
In terms of the gap lengths, this is precisely the right-hand side of
\eqref{eq:equality}.  Thus Theorem~\ref{thm:equality} gives another proof of
Conjecture~3.20 of \cite{FloriRandriamCohenMesnager2010}.
\end{rem}

\subsubsection{Counting the equality parameters}

Let $E_{k,Z}$ denote the number of equality parameters whose $k$-bit
word has exactly $Z$ zeros.

\begin{cor}\label{cor:enum}
If $1\le Z\le\lfloor\sqrt{k}\rfloor$, then
\begin{equation}\label{eq:E-kZ}
 E_{k,Z}
 =
 \frac{k}{Z}
 \binom{k-Z^2+Z-1}{Z-1}.
\end{equation}
For $Z>\sqrt{k}$ one has $E_{k,Z}=0$.  Hence
\begin{equation}\label{eq:E-k}
 E_k
 =
 \sum_{1\le Z\le\lfloor\sqrt{k}\rfloor}
 \frac{k}{Z}
 \binom{k-Z^2+Z-1}{Z-1}.
\end{equation}
\end{cor}

\begin{proof}
Under equality write
\[
 g_i=(Z-1)+h_i,
 \qquad
 h_i\ge0.
\]
Since the $Z$ gaps contain all $R=k-Z$ ones,
\[
 h_1+\cdots+h_Z
 =
 k-Z-Z(Z-1)
 =
 k-Z^2.
\]
Thus equality requires $k\ge Z^2$, and the number of ordered gap vectors with a distinguished zero is
\[
 \binom{k-Z^2+Z-1}{Z-1}.
\]

Count pairs consisting of an equality word and a distinguished zero.
Choosing the absolute position of the distinguished zero gives $k$
choices, and then the ordered gap vector determines the word
uniquely.  On the other hand every word has exactly $Z$ choices of
distinguished zero.  Dividing by $Z$ proves
\eqref{eq:E-kZ}, and summing gives \eqref{eq:E-k}.
\end{proof}

\begin{rem}
The equality criterion and the enumeration formula above were also
obtained by Yuan et al.\
\cite[Theorem~3.1.2 and Corollary~3.1.1]{YuanEtAl2026}, using a
different integral argument.
\end{rem}

\subsection{A lower bound for the positive deficit when $R\ge Z$}

\begin{thm}\label{thm:deficit-bound}
Assume $R\ge Z\ge2$, and suppose that $t$ is not an equality
parameter.  Then
\begin{equation}\label{eq:deficit-bound}
 2^{k-1}-|S_{t,k}|
 \ge
 2^{R-Z+1}.
\end{equation}
Moreover, equality holds in \eqref{eq:deficit-bound} if and only if one of
the following occurs:
\begin{enumerate}[label=\textup{(\roman*)},leftmargin=2.8em]
\item
exactly one cyclic gap has length $Z-2$, and all the remaining cyclic
gaps have length at least $Z-1$;
\item
$Z=3$, exactly one cyclic gap has length $0$, and the other two cyclic
gaps have length at least $2$.
\end{enumerate}
\end{thm}

\begin{proof}
Put $m=Z-1$.  By \eqref{eq:deficit-factor-plus},
\[
 2^{k-1}-|S_{t,k}|
 =
 2^{R-Z+1}
 \sum_{\ell=1}^{m}4^{m-\ell}\Delta_\ell(W).
\]
The quantities $\Delta_\ell(W)$ are nonnegative integers.  Since a
nonequality parameter has at least one positive $\Delta_\ell(W)$, \eqref{eq:deficit-bound} follows, and
equality holds if and only if
\begin{equation}\label{eq:deficit-bound-profile}
 \Delta_m(W)=1,
 \qquad
 \Delta_\ell(W)=0
 \quad(1\le\ell<m).
\end{equation}

We first classify the condition $\Delta_m(W)=1$.  For
$\mathbf g=(g_0,\ldots,g_m)$ put
\[
 \Fcal_q(\mathbf g)
 =
 \left\{
 \mathbf x\in\N^{m+1}:
 |\mathbf x|=q,\ x_i\le g_i\ \text{for every }i
 \right\}.
\]
Since $R\ge m+1$, Theorem~\ref{thm:top-box} gives
\[
 \Delta_m(W)
 =
 2|\Fcal_{m-1}(\mathbf g)|
 -
 |\Fcal_m(\mathbf g)|.
\]
For $\mathbf x\in\Fcal_{m-1}(\mathbf g)$ consider all marked insertions
of one additional $1$.  There are $2m$ insertion slots.  The
$x_i+1$ slots belonging to the $i$th gap are admissible precisely when
$x_i<g_i$.  On the other hand, every word represented by
$\Fcal_m(\mathbf g)$ has $m$ marked deletions.  Hence
\begin{equation}\label{eq:marked-count}
 m\Delta_m(W)
 =
 \sum_{\substack{0\le i\le m\\g_i\le m-1}}
 (g_i+1)N_i,
\end{equation}
where
\begin{equation}\label{eq:Ni}
 N_i
 =
 [y^{m-1-g_i}]
 \prod_{j\ne i}(1+y+\cdots+y^{g_j}).
\end{equation}

We use the elementary fact that a product of polynomials
$1+y+\cdots+y^a$ has a symmetric unimodal coefficient sequence.  This
follows inductively from convolution.  Let
\[
 r=|\{i:g_i>0\}|.
\]
Suppose first that $m\ge3$ and $2\le r\le m$.  For a zero gap $g_i=0$,
the polynomial in \eqref{eq:Ni} has $r$ positive factors and total degree
$R$.  Its coefficient sequence is symmetric and unimodal.  Since
$2\le m-1\le R-2$, both $m-1$ and its symmetric degree
$R-(m-1)$ are at least $2$; hence unimodality gives
\[
 N_i
 \ge
 [y^2]\prod_{g_j>0}(1+y+\cdots+y^{g_j})
 =
 \binom r2+|\{j:g_j\ge2\}|
 \ge
 \binom r2+1,
\]
because $R\ge m+1>r$.  Since there are $m+1-r$ zero gaps, their
contribution to \eqref{eq:marked-count} is at least
\[
 (m+1-r)\left(\binom r2+1\right).
\]
For $r=2$ this is $2(m-1)>m$.  For $r\ge3$ it is at least
\[
 (m+1-r)(r+1)
 =
 m+r(m-r)+1
 >
 m.
\]
Thus $\Delta_m(W)>1$.

If $r=1$, exactly $m$ gaps are zero and the remaining gap has length
$R\ge m+1$.  Each zero gap contributes $1$ to
\eqref{eq:marked-count}, so $\Delta_m(W)=1$.

Now suppose that $m\ge3$ and $r=m+1$, so all gaps are positive.  If
$g_i=m-1$, then $N_i=1$, and the $i$th term of
\eqref{eq:marked-count} is exactly $m$.  Hence
$\Delta_m(W)=1$ is possible only if this is the unique gap with $g_i<m$
and every other gap is at least $m$.  If instead
$1\le g_i\le m-2$, then the polynomial in \eqref{eq:Ni} has $m$
positive factors and, by symmetry and unimodality,
\[
 N_i\ge
 [y]\prod_{j\ne i}(1+y+\cdots+y^{g_j})=m.
\]
Thus the $i$th term alone is at least $2m$, which is impossible when
$\Delta_m(W)=1$.

When $m=2$, the same argument leaves one additional possibility.  If
exactly one gap is zero and the other two are positive, then for the
zero gap
\[
 N_i
 =
 [y](1+y+\cdots+y^{g_j})(1+y+\cdots+y^{g_k})
 =
 2.
\]
It contributes exactly $2=m$ to
\eqref{eq:marked-count}.  Thus $\Delta_2(W)=1$ occurs in this
case precisely when the other two gaps are at least $2$.  For $m=1$,
formula \eqref{eq:marked-count} shows directly that
$\Delta_1(W)=1$ precisely when exactly one of the two cyclic gaps is
zero.

We have therefore shown that $\Delta_m(W)=1$ can occur only in the
following three situations:
\[
\begin{split}
&\text{one gap is }m-1\text{ and all the others are at least }m;\\
&m=2,\ \text{one gap is }0\text{ and the other two are at least }2;\\
&m\ge2,\ \text{all but one gap are }0.
\end{split}
\]
The last situation cannot satisfy \eqref{eq:deficit-bound-profile}.  Indeed,
then all ones of $W$ form a single cyclic run, and after a cyclic
rotation
\[
 W=10v,
 \qquad
 v=0^m1^{R-1}.
\]
Thus $01\in D(v)$, whereas $10\notin D(v)$ and
$10>_{\rm lex}v$, so $10\notin\Lcal(v)$.  Hence
$\Lcal(v)\cap\W_{1,1}$ is nonempty but not full.  By
Theorem~\ref{thm:deletion-equality}, $\Delta_1(W)>0$, contradicting
\eqref{eq:deficit-bound-profile}.

It remains to verify that the two stated families attain the bound.
For (i), choose the cyclic rotation $W=10v$ so that the unique gap of
length $m-1$ does not cross the cut.  The set in Theorem~\ref{thm:top-box} is then the full degree-$m$
level with exactly one vector $m e_i$ removed, while the degree-$(m-1)$
level is full.  Hence $\Delta_m(W)=1$.  Every bound relevant to a level $\ell<m$ is at least
$m-1\ge\ell$, so
Lemma~\ref{lem:coarse} shows that $\Lcal(v)\cap\W_{\ell,\ell}$ is full for every
$\ell<m$; deleting one $1$ gives the corresponding level
$\Lcal(v)\cap\W_{\ell,\ell-1}$ as well.  Thus
$\Delta_\ell(W)=0$ for all $\ell<m$.

For (ii), $m=2$ and, up to cyclic rotation, the gap vector is
$(0,a,b)$ with $a,b\ge2$.  We may write
\[
 W=10v,
 \qquad
 v=1^b00\,1^{a-1}.
\]
Both $01$ and $10$ occur as subsequences of $v$, so
\[
 \Lcal(v)\cap\W_{1,1}=\W_{1,1},
 \qquad
 \Lcal(v)\cap\W_{1,0}=\W_{1,0}.
\]
Thus $\Delta_1(W)=0$, while the classification above gives
$\Delta_2(W)=1$.  Hence \eqref{eq:deficit-bound-profile} holds.  This
proves the theorem.
\end{proof}

\subsection{Further results when $R<Z$}

For a cyclic binary word $W$, let $\overline{W}$ denote its bitwise
complement.

\begin{lem}\label{lem:complement-symmetry}
For every cyclic binary word $W$,
\[
 H_{\overline{W}}(u,v)=H_W(v,u).
\]
Equivalently,
\[
 h_{p,q}(\overline{W})=h_{q,p}(W)
 \qquad(p,q\ge0).
\]
\end{lem}

\begin{proof}
Recall
\[
 P=
 \begin{pmatrix}
 0&1\\
 1&0
 \end{pmatrix}.
\]
A direct calculation gives
\[
 C_{1-b}(u,v)=P C_b(v,u)P
 \qquad(b\in\{0,1\}).
\]
Hence
\[
 M_{\overline{W}}(u,v)=P M_W(v,u)P,
\]
and therefore
\[
 \operatorname{tr}M_{\overline{W}}(u,v)
 =
 \operatorname{tr}M_W(v,u).
\]
Complementation interchanges the numbers of zeros and ones, so comparing
both sides with \eqref{eq:Cusick-H} gives
\[
 H_{\overline{W}}(u,v)=H_W(v,u).
\]
The coefficient identity follows immediately.
\end{proof}

Assume now that $R<Z$. Let
$z_1,\ldots,z_R$ be the cyclic numbers of zeros between successive ones of $W$, and put
\[
 s=|\{i:z_i>0\}|.
\]
Thus $s$ is the number of cyclic runs of $1$'s in $W$ and
\[
 z_1+\cdots+z_R=Z.
\]

\begin{thm}\label{thm:low-weight}
Assume $R<Z$.  Then the following statements hold.
\begin{enumerate}[label=\textup{(\roman*)},leftmargin=2.8em]
\item
There is a positive integer $M_-(W)$ such that
\begin{equation}\label{eq:low-arithmetic}
 2^{k-1}-|S_{t,k}|
 =
 2^{Z-R}M_-(W)-1,
\end{equation}
where explicitly
\begin{equation}\label{eq:M-minus}
 M_-(W)
 =
 h_{R-1,R}(W)
 +
 2\sum_{\ell=1}^{R-1}
 4^{R-\ell-1}\Delta_\ell(W).
\end{equation}
In particular,
\begin{equation}\label{eq:low-congruence}
 |S_{t,k}|\equiv1\pmod{2^{Z-R}}.
\end{equation}

\item
The quantity $\Delta_R(W)$ satisfies
\begin{equation}\label{eq:low-coeff}
 \Delta_R(W)
 =
 2[y^R]
 \prod_{i=1}^{R}(1+y+\cdots+y^{z_i}).
\end{equation}

\item
If $s$ is the number of cyclic runs of $1$'s in $W$, then
\begin{equation}\label{eq:low-run-bound}
 2^{k-1}-|S_{t,k}|
 \ge
 2^{Z-R}
 \sum_{r=0}^{\min\{Z-R,s-1\}}
 \binom{s-1}{r}
 -1.
\end{equation}
\end{enumerate}
\end{thm}

\begin{proof}
As in the proof of Proposition~\ref{prop:deficit}, the bound
$\deg_YJ_v\le R-1$ gives
\[
 j_{R,R}(v)=0.
\]
Moreover, $v$ has $R-1$ ones and $Z-1\ge R$ zeros, so $D(v)$ contains
at least one word with $R$ zeros and $R-1$ ones. Hence
$j_{R,R-1}(v)\ge1$.
By Theorem~\ref{thm:bridge},
\[
 h_{R-1,R}(W)=j_{R,R-1}(v),
 \qquad
 h_{R,R}(W)=j_{R,R}(v)=0,
\]
and therefore
\begin{equation}\label{eq:last-term}
 \Delta_R(W)=2h_{R-1,R}(W).
\end{equation}
Using Corollary~\ref{cor:divisibility}, we obtain
\begin{align*}
 2^{k-1}-|S_{t,k}|+1
 =
 2^{Z-R-1} \sum_{\ell=1}^{R}4^{R-\ell}\Delta_\ell(W)=
 2^{Z-R}
 \left(
 h_{R-1,R}(W)
+2\sum_{\ell=1}^{R-1}4^{R-\ell-1}\Delta_\ell(W)
 \right).
\end{align*}
The expression in parentheses is the positive integer $M_-(W)$ in
\eqref{eq:M-minus}, proving \eqref{eq:low-arithmetic} and
\eqref{eq:low-congruence}.

We next prove \eqref{eq:low-coeff}.  If $R=1$, then after choosing
$W=10v$ one has $v=0^{Z-1}$, so $j_{1,0}(v)=1$ and $j_{1,1}(v)=0$;
hence $\Delta_1(W)=2$, which is exactly \eqref{eq:low-coeff}.
Assume $R\ge2$.  The complemented word $\overline{W}$ has $Z$ ones and
$R$ zeros.  Its gap lengths are precisely
$z_1,\ldots,z_R$.  Applying Theorem~\ref{thm:top-box} to
$\overline{W}$ with highest level $R-1$ and with $q=R<Z$ gives
\[
 h_{R,R-1}(\overline{W})
 =
 [y^R]\prod_{i=1}^{R}(1+y+\cdots+y^{z_i}).
\]
Lemma~\ref{lem:complement-symmetry} gives
\[
 h_{R,R-1}(\overline{W})=h_{R-1,R}(W).
\]
Combining this with \eqref{eq:last-term} proves
\eqref{eq:low-coeff}.

It remains to prove \eqref{eq:low-run-bound}.  Write
\[
 S_a(y)=1+y+\cdots+y^a.
\]
Exactly $s$ of the zero-gaps are positive.  If $2\le a\le b$, then
\begin{equation}\label{eq:gap-compression}
 S_a(y)S_b(y)-S_{a-1}(y)S_{b+1}(y)
 =
 y^a+y^{a+1}+\cdots+y^b,
\end{equation}
which has nonnegative coefficients.  Repeatedly applying
\eqref{eq:gap-compression} to the positive zero-gaps shows coefficientwise
that
\[
 \prod_{i=1}^{R}S_{z_i}(y)
 \succeq
 S_{Z-s+1}(y)(1+y)^{s-1},
\]
where $\succeq$ denotes coefficientwise domination; that is,
$A(y)\succeq B(y)$ means that every coefficient of $A(y)-B(y)$
is nonnegative.  Put $d=Z-R$.
Since $s\le R$, coefficient extraction gives
\begin{align*}
 [y^R]S_{Z-s+1}(y)(1+y)^{s-1}
 =
 \sum_{r=\max\{0,s-1-d\}}^{s-1}
 \binom{s-1}{r}=
 \sum_{r=0}^{\min\{d,s-1\}}
 \binom{s-1}{r}.
\end{align*}
Thus \eqref{eq:low-coeff} yields
\[
 \Delta_R(W)
 \ge
 2\sum_{r=0}^{\min\{Z-R,s-1\}}
 \binom{s-1}{r}.
\]
Finally, all $\Delta_\ell(W)$ are nonnegative, and hence
\[
 2^{k-1}-|S_{t,k}|+1
 =
 2^{Z-R-1}
 \sum_{\ell=1}^{R}4^{R-\ell}\Delta_\ell(W)
 \ge
 2^{Z-R-1}\Delta_R(W).
\]
This proves \eqref{eq:low-run-bound}.
\end{proof}

\begin{rem}\label{rem:low-special-cases}
Theorem~\ref{thm:low-weight} gives, in particular,
\[
 2^{k-1}-|S_{t,k}|
 \ge
 2^{Z-R}-1,
\]
so equality in the Tu--Deng bound is impossible when $R<Z$.  More
precisely, if $Z=R+1$, then
\[
 2^{k-1}-|S_{t,k}|\ge2s-1,
\]
while if $Z-R\ge s-1$, then
\[
 2^{k-1}-|S_{t,k}|
 \ge
 2^{Z-R+s-1}-1.
\]
\end{rem}

\begin{cor}\label{cor:div}
For $R\ge Z$,
\[
 2^{R-Z+1}
 \mid
 \bigl(2^{k-1}-|S_{t,k}|\bigr).
\]
For $R<Z$,
\[
 2^{Z-R}
 \mid
 \bigl(2^{k-1}-|S_{t,k}|+1\bigr).
\]
\end{cor}

\begin{proof}
The first assertion is Corollary~\ref{cor:divisibility}.  The second is
\eqref{eq:low-arithmetic}.
\end{proof}

\section{Concluding remarks}

We have given a combinatorial proof of the equality criterion for the
Tu--Deng bound: if the cyclic binary word of $t$ contains $Z$ zeros,
then equality holds exactly when every run of ones between consecutive
zeros has length at least $Z-1$.  This is the condition conjectured by
Flori, Randriambololona, Cohen and Mesnager and also proved,
together with the same enumeration formula, by Yuan et al.\
\cite{YuanEtAl2026}.

The coefficient description also gives information beyond equality.
When $R\ge Z$, the deficit $2^{k-1}-|S_{t,k}|$ is divisible by
$2^{R-Z+1}$.  For $R\ge Z\ge2$, every nonequality parameter satisfies
\[
 2^{k-1}-|S_{t,k}|\ge2^{R-Z+1},
\]
and Theorem~\ref{thm:deficit-bound} classifies all parameters for which
this lower bound is attained.  When $R<Z$,
Theorem~\ref{thm:low-weight} gives
\[
 2^{k-1}-|S_{t,k}|=2^{Z-R}M_-(W)-1
\]
with $M_-(W)>0$, together with a lower bound depending on the number of
cyclic runs of ones.

Determining the exact minimum positive deficit for fixed $R$ and $Z$
in the cases not covered by the equality statement in
Theorem~\ref{thm:deficit-bound}, and in particular minimizing $M_-(W)$
when $R<Z$, remains open.


\begin{thebibliography}{99}

\bibitem{Cheng2026}
K. Cheng,
A first-exit proof of Cusick's sum-of-digits conjecture,
arXiv:2606.23398, 2026.

\bibitem{ChengHongZhong2015}
K. Cheng, S. Hong and Y. Zhong,
A note on the Tu--Deng conjecture,
\emph{J. Syst. Sci. Complex.} \textbf{28} (2015), no.~3, 702--724,
DOI 10.1007/s11424-015-2240-3.

\bibitem{ChenLinWei2020}
Y. Chen, L. Lin and C. Wei,
About the Tu--Deng conjecture for $w(t)$ less than or equal to $10$,
\emph{IACR Cryptol. ePrint Arch.} 2020, Paper No.~227.

\bibitem{Cusick2026}
T. W. Cusick,
Proof of the Tu--Deng conjecture,
arXiv:2608.14821, 2026.

\bibitem{CusickLiStanica2011}
T. W. Cusick, Y. Li and P. St\u{a}nic\u{a},
On a combinatorial conjecture,
\emph{Integers} \textbf{11} (2011), Paper No.~A17.

\bibitem{DengYuan2012}
G. Deng and P. Yuan,
On a combinatorial conjecture of Tu and Deng,
\emph{Integers} \textbf{12} (2012), Paper No.~A48.

\bibitem{FloriRandriamCohenMesnager2010}
J.-P. Flori, H. Randriambololona, G. Cohen and S. Mesnager,
On a conjecture about binary strings distribution,
in: \emph{Sequences and Their Applications--SETA 2010},
Lecture Notes in Comput. Sci. \textbf{6338},
Springer, Berlin, 2010, pp.~346--358,
DOI 10.1007/978-3-642-15874-2\_30.

\bibitem{LiuLuoXie2026}
R. Liu, H. Luo and T. Xie,
A complete proof for Tu--Deng conjecture,
arXiv:2608.05187, 2026.

\bibitem{SpiegelhoferWallner2019}
L. Spiegelhofer and M. Wallner,
The Tu--Deng conjecture holds almost surely,
\emph{Electron. J. Combin.} \textbf{26} (2019), no.~1,
Paper No.~P1.28.


\bibitem{YuanEtAl2026}
Y. Yuan, Y. Hu, Y. Zhang, L. Zhang and W. Wu,
Theoretical open problems in symmetric cryptography: verifiable
LLM-guided analysis,
\emph{IACR Cryptol. ePrint Arch.} 2026, Paper No.~1687.

\bibitem{TuDeng2011}
Z. Tu and Y. Deng,
A conjecture about binary strings and its applications on constructing
Boolean functions with optimal algebraic immunity,
\emph{Des. Codes Cryptogr.} \textbf{60} (2011), no.~1, 1--14.

\end{thebibliography}
\end{document}